\documentclass[20pt]{amsart}
\usepackage{amsfonts}
\usepackage{graphicx}
\usepackage{amscd}
\usepackage{color}
\usepackage{amsmath,amsfonts,amssymb,amsthm,mathrsfs,xypic,cite,leftidx,enumitem}
\xyoption{curve}
\usepackage{fancybox}
\usepackage{tikz}
\usepackage{tikz-cd}
\usetikzlibrary{arrows,decorations.pathmorphing,backgrounds,positioning,fit,petri}
\usetikzlibrary{arrows,patterns,matrix}
\usetikzlibrary{decorations.markings}
\usepackage[all,cmtip]{xy}
\usepackage{hyperref}

\newcommand{\m}{\Lambda}
\newcommand{\Hom}{\operatorname{Hom}}

\newcommand{\Ker}{\operatorname{Ker}}
\newcommand{\cok}{\operatorname{Coker}}
\newcommand{\Ima}{\operatorname{Im}}
\newcommand{\rad}{\operatorname{rad}}
\newcommand{\Ext}{\operatorname{Ext}}
\newcommand{\add}{\operatorname{add}}
\newcommand{\op}{\operatorname{op}}

\DeclareMathOperator{\pd}{proj.dim}%
\DeclareMathOperator{\id}{inj.dim}%

\newcommand{\ha}{\operatorname{\mathcal{A}}}
\newcommand {\hp}{\mathcal {P}}

\newcommand{\Gp}{{\rm Gproj}}
\newcommand{\fX}{\mathcal{X}}
\newcommand{\fY}{\mathcal{Y}}

\newcommand{\rep}{{\rm rep}}
\newcommand{\smon}{{\rm smon}}
\newcommand{\sepi}{{\rm sepi}}

\newcommand{\Proj}{{\rm proj}}

\newtheorem{thm}{Theorem}[section]

\newtheorem{cor}[thm]{Corollary}

\newtheorem{lem}[thm]{Lemma}
\newtheorem{exm}[thm]{Example}

\newtheorem{prop}[thm]{Proposition}
\newtheorem{rem}[thm]{Remark}
\newtheorem{defn}[thm]{Definition}
\newtheorem{constr}[thm]{Construction}
\newtheorem{que}[thm]{Question}
\newcounter{margincounter}
\begin{document}

\title [Semi-Gorenstein-projective ]
{ Separated monic correspondence of cotorsion pairs and Semi-Gorenstein-projective modules}
\author [Xiu-Hua Luo, Shijie Zhu ] {Xiu-Hua Luo, Shijie Zhu$^*$}
\thanks{$^*$The corresponding author.}
\dedicatory{Dedicated to Professor Pu Zhang on the occasion of
  his 60th birthday}
\thanks{Supported by the NSF of China (No. 12201321).}
\thanks{xiuhualuo$\symbol{64}$ntu.edu.cn  \ \ \ \ shijiezhu$\symbol{64}$ntu.edu.cn}
 \maketitle

\begin{center}
Department of Mathematics, \ \ Nantong University\\
Jiangsu 226019, P. R. China
\end{center}

\begin{abstract} \ Given a finite dimensional algebra $A$ over a field $k$, and a finite acyclic quiver $Q$,
let $\Lambda = A\otimes_k kQ/I$, where $kQ$ is the path algebra of $Q$ over $k$ and $I$ is a monomial ideal.
We show that $(\fX,\fY)$ is a (complete) hereditary cotorsion pair in $A$-mod if and only if $(\smon(Q,I,\fX), \rep(Q,I,\fY))$ is a (complete) hereditary cotorsion pair in $\m$-mod.
We also show that $A$ is left weakly Gorenstein if and only if so is $\m$.  Provided that $kQ/I$ is non-semisimple, the category $\leftidx^{\perp}\m$ of semi-Gorenstein-projective $\m$-modules coincides with the category of separated monic representations $\smon(Q,I,\leftidx^{\perp}A)$ if and only if $A$ is left weakly Gorenstein.

\vskip5pt


\end{abstract}

\section {\bf Introduction}

\vskip10pt

{\bf Throughout this paper, $k$ is a field and $A$ is a finite dimensional $k$-algebra. Let $Q$ be a finite acyclic quiver, $I$ an admissible ideal of the path algebra $kQ$ generated by monomial relations. Let $\m=A\otimes_k kQ/I$. Denote by $A$-mod ($\m$-mod) the category of finitely generated left $A$-modules ($\m$-modules). Denote by $\mathcal C^\perp:=\bigcap\limits_{i=1}^\infty\ker\Ext^i(\mathcal C,-)$ and    $^\perp\mathcal C:=\bigcap\limits_{i=1}^\infty\ker\Ext^i(-,\mathcal C)$        .}

\subsection{Background}

The notion of cotorsion classes in abelian categories was introduced by Salce in 1979 \cite{Sa}. A pair of subcategories $(\fX, \fY)$ in an abelian category $\ha$ is called a {\it cotorsion pair} if $\fX=\ker \Ext^1(-,\fY)$ and $\fY=\ker\Ext^1(\fX,-)$, where $\fX$ is called a {\it cotorsion class} and $\fY$ a {\it cotorsion-free class}. The definition of cotorsion pairs is similar to torsion pairs, in the sense that $\Ext^1$-orthoganality is considered for the pair of subcategories instead of $\Hom$-orthoganality. Cotorsion pairs have great significance in approximation theory as Salce's Lemma shows that under suitable conditions, every object in $\ha$ has a right $\fX$-approximation if and only if every object in $\ha$ has a left  $\fY$-approximation. Such cotorsion pairs are called {\it complete}. Later on, an important class of complete cotorsion pairs was investigated in Auslander and Reiten's fundamental paper \cite{AR1}, it was proved that in a cotorsion pair $(\fX,\fY)$, the subcategory $\fX$ is contravariantly finite resolving if and only if $\fY$ is covariantly finite coresolving. Such pairs $(\fX,\fY)$ are called {\it complete hereditary} cotorsion pairs. Moreover, Auslander and Reiten also established a correspondence between cotilting modules and certain complete hereditary cotorsion pairs (see Theorem \ref{cotilting induce} below). Our paper is devoted to study the correspondence between (complete) hereditary cotorsion pairs in $A$-mod and (complete) hereditary cotorsion pairs in $\m$-mod under the correspondence $(\fX,\fY)\mapsto (\smon(Q,I,\fX),\rep(Q,I,\fY))$, where $\smon(Q,I,\fX)$ is the subcategory of separated monic representations over $\fX$ (see Definition \ref{maindef}) which was originally invented for describing Gorenstein-projective $\m$-modules.

{\it A complete $A$-projective
resolution} is an exact sequence of finitely generated projective
$A$-modules
$$P^\bullet = \cdots \longrightarrow P^{-1}\longrightarrow P^{0}
\stackrel{d^0}{\longrightarrow} P^{1}\longrightarrow \cdots$$ such
that ${\rm Hom}_A(P^\bullet, A)$ is also exact. A module $M\in
A$-mod is {\it Gorenstein-projective}, if there exists a complete
$A$-projective resolution $P^\bullet$ such that $M\cong
\operatorname{Ker}d^0$ \cite{EJ1}.
This class of modules enjoys more stable properties than
the usual projective modules \cite{AB}. It becomes an important topic in the study of relative
homological algebra \cite{EJ1, EJ2} and in the representation theory
of algebras (see \cite{AR1, AR2, B, GZ, IKM}). From the geometric aspects, Gorenstein-projective modules are also closely related with singularity categories \cite{Buch, Hap2}.
   Thus explicitly constructing all the Gorenstein-projective
modules over a given algebra is a fundamental problem.

Recently, various methods have been developed in the philosophy of constructing Gorenstein-projective modules over a large algebra from those over a smaller algebra. For instance, in \cite{LZ, XZ} the authors characterized Gorenstein-projective modules over triangular matrix rings. Another efficient method of describing Gorenstein-projective modules involves the submodule category, which dates back to G. Birkhoff \cite{Birk} and has been extensively studied  in \cite{RS1, RS2, RS3, Si}. It turns out that the category of the (strongly) Gorenstein-projective modules is
closely related to the submodule category  \cite{LZ,XZ}. Motivated by the relation between submodule categories and the construction of Gorenstein-projective modules over triangular matrix rings, more general methods of constructing Gorenstein-projective modules over tensor algebras have been recently developed \hspace{1sp}\cite{Z, LZ2, LZ3, HLXZ, JL}, using monomorphism categories.
For instance,  we will use the following result:

\begin{thm}\cite[Theorem 4.1]{LZ3}\label{LZ3 thm}
Let $Q$ be an acyclic quiver, $I$ an monomial ideal of $kQ$ and $A$ a finite dimensional algebra over a field $k$. Let $\m=A\otimes_kkQ/I$, and $X=(X_i,X_\alpha, i\in Q_0, \alpha\in Q_1)$ be a $\m$-module. Then $X\in\Gp(\m)$ if and only if $X$ is separated monic and   $X_i/\bigoplus \limits_{\begin
{smallmatrix} \alpha\in Q_1\\ e(\alpha) = i \end{smallmatrix}} \Ima X_\alpha$ lies in $\Gp(A)$ for each $i\in Q_0$.
\end{thm}

This theorem suggests a correspondence between the class of Gorenstein-projective $A$-modules and the class of Gorenstein-projective $\m$-modules, given by $\Gp(A)\to\smon(Q,I,\Gp(A))$.
In fact, such a correspondence $\fX\mapsto \smon(Q,I,\fX)$ has also been studied in \cite{Z} and \cite{ZX} for $\fX=\leftidx{^\perp}T$ for some cotilting $A$-module $T$. The authors discovered a reciprocity of the separated monic representations: $\smon(Q,I,\leftidx{^\perp}T)= \leftidx{^\perp}(T\otimes kQ/I)$. Following by a well-known result from Auslander-Reiten (see Theorem \ref{cotilting induce}), it shows that $\fX\to \smon(Q,I,\fX)$ yields a correspondence between contravariantly finite resolving cotorsion classes $\fX$ satisfying $\widehat\fX=A$-mod and   contravariantly finite resolving cotorsion classes $\smon(Q,I, \fX)$ satisfying $\widehat {\smon(Q,I, \fX)} =\m$-mod.

 In a weaker sense, a module $M$ is said to be {\it semi-Gorenstein-projective} provided that $M\in \leftidx{^\perp A}=\{X|\Ext^i
(X, A) = 0$ for all $ i \geq 1\}$ \cite{RZ2}.
These modules have been called Cohen-Macaulay
modules \cite{B, Buch} or stable modules  \cite{AB}  in different contexts \cite{JS, Mar}.
All Gorenstein-projective modules are semi-Gorenstein-projective.
 Let $\Proj(A)$ be the full
subcategory of finitely generated projective $A$-modules,  $\Gp(A)$
the full subcategory of finitely generated Gorenstein-projective $A$-modules. Then $\Proj(A) \subseteq \Gp(A)\subseteq \leftidx{^\perp}A  $.
If $A$ is {\it a
Gorenstein algebra} (i.e., ${\rm inj.dim}\ _AA< \infty$ and ${\rm
inj.dim} \ A_A < \infty$) then \ $\Gp(A) = \leftidx{^\perp} A$
\cite{EJ2}. However, the converse  is {\it not} true.
In \cite{RZ2}, the artin algerba is called  {\it a  left (right) weakly Gorenstein algebra} provided $\Gp(A) = \ \leftidx{^\perp}\!A$
, i.e. any semi-Gorenstein-projective left (right) module is Gorenstein-projective. Note that the class of left weakly Gorenstein algebras is much larger than the class of Gorenstein algebras \cite[\S 1.9]{RZ2}. It was showed that torsionless finite algebras are left and right weakly Gorenstein \cite{Mar2}, which includes all representation finite type algebras, algebras stable equivalent to hereditary algebras and minimal representation infinite algebras \cite{Rin}.

In \cite{Mar2}, the author showed that every left weakly Gorenstein algebra satisfies the strong Nakayama conjecture, and therefore the  Nakayama conjecture. Thus  to prove the Nakayama conjecture, one needs to check that it holds for non-left weakly Gorenstein algebra.
However, the problems of constructing semi-Gorenstein-projective modules which are not Gorenstein-projective, or finding (non-)left weakly Gorenstein algebras are not easy. In 2006, the first example was found by Jorgensen and $\c{S}$ega \cite{JS}, about 40 years after the definition of Gorenstein-projective modules.
In 2012, similar to the construction of Gorenstein-projective modules over triangular matrix rings, Xiong and Zhang developed a method of consctructing semi-Gorenstein-projective modules over triangular matrix rings.

\begin{thm}\cite[Theorem 1.1]{XZ}\label{matrixring_thm}
Let $\m=\begin{bmatrix}B&_BM_A\\0&A\end{bmatrix}$ be an artin algebra. Assume that $\pd _BM<\infty$ and $D(M_A)\in(\leftidx{^\perp}A)^\perp$. Then $\begin{bmatrix}X\\Y\end{bmatrix}_\phi\in\leftidx{^\perp}\m$ if and only if $\phi^*:\Hom_B(X,B)\to \Hom_B(M\otimes_A Y,B)$ is an epimorphism, $\phi$ induces isomorphisms $\Ext^i_B(M\otimes_AY,B)\cong\Ext^i_B(X,B)$ for all $i\geq 1$ and $Y\in \leftidx{^\perp}A$.
\end{thm}

So far, we have limited knowledge about determining the left weakly Gorensteinness of an algebra.
In \cite[Theorem 1.3]{RZ2}, the authors proved  that if the number of isomorphism classes of indecomposable left $A$-modules which are both semi-Gorenstein-projective and torsionless is finite, then $A$ is left weakly Gorenstein. In a recent paper \cite {Z3}, the author showed the following result for determining whether a triangular matrix ring is left weakly Gorenstein.

\begin{thm}\cite[Proposition 4.1]{Z3}\label{Z3 thm}
Let $A$, $B$ be artin algebras and $\m=\begin{bmatrix}B&_BM_A\\0&A\end{bmatrix}$.  Assume that $\pd _BM<\infty$, $\pd M_A<\infty$ and $D(M_A)\in(\leftidx{^\perp}A)^\perp$. Then $\m$ is left weakly Gorenstein if and only if each semi-Gorenstein-projective $\m$-module is monic and both $A$ and $B$ are left weakly Gorenstein.
\end{thm}

In this paper, we will use  Theorem \ref{matrixring_thm} to characterize semi-Gorenstein-projective modules over certain triangular matrix algebras. It turns out for the triangular matrix ring $\m$ satisfying the assumptions in Theorem \ref{matrixring_thm}, if $A$ and $B$ are left weakly Gorenstein, then each semi-Gorenstein-projective $\m$-module is monic (see Lemma \ref{matrixring lwg_lem}). Hence, we can improve Theorem \ref{Z3 thm} by dropping the condition ``each semi-Gorenstein-projective $\m$-module is monic'' in the ``if part''. As an application, we show that for $\m=A\otimes kQ/I$,  $A$ is left weakly Gorenstein if and only if so is $\m$, which leads to a method of creating new left weakly Gorenstein, or non-left weakly Gorenstein algebras from the algebra $A$.

\vskip10pt

\subsection{Outline of the paper}

\begingroup
\setcounter{thm}{0}
\renewcommand\thethm{\Alph{thm}}

In section 2, we recall some preliminaries of cotorsion pairs and resolving subcategories. In section 3, we give a brief introduction to separated monic representations over a subcategory $\fX$ and their filtration interpretations. Our study of two adjoint pairs $(\cok_i, -\otimes S(i))$ and $(-\otimes P(i), (-)_i)$ provides a main tool (Lemma \ref{adj ext_lem}) for the rest of the paper.

 In section 4, we will study the separated monic correspondence
$\fX\to\smon(Q,I,\fX)$ for (contravariantly finite) resolving cotorsion classes $\fX$, and show the following theorem:

\begin{thm}\label{thm A}  Let $A$ be
a finite dimensional algebra, $Q$ a finite acyclic quiver, $I$ an
admissible ideal of the path algebra $kQ$ generated by monomial relations and $\m=A\otimes_kkQ/I$.
  Then
$(\fX,\fY)$ is a hereditary cotorsion pair in $A$-mod if and only if $({\rm smon}(Q, I,\fX ),{\rm rep}(Q,I,\fY))$ is a hereditary cotorsion pair in $\m$-mod.
Furthermore, $(\fX,\fY)$ is a complete hereditary cotorsion pair if and only if so is $({\rm smon}(Q, I,\fX ),{\rm rep}(Q,I,\fY))$.
\end{thm}

 In summary, the correspondence $\fX \to {\rm smon}(Q, I, \fX )$ preserves the following subcategories:
(a) \{resolving subcategories\} (b)\{resolving cotorsion classes\} (c)\{contravariantly finite resolving subcategories\} (d)\{cotorsion classes $\fX$, with $\hat \fX=A$-mod\}.

In section 5, we study the separated monic representation of the subcategory of semi-Gorenstein-projective $A$-modules and compare it with the subcategory of semi-Gorenstein-projective $\m$-modules. We show the equality $\smon(Q,I,\leftidx{^\perp}A)=\leftidx{^\perp}\m$ holds when $A$ is left weakly Gorenstein.

\begin{thm}\label{thm C} Let $A$ be
a finite dimensional algebra, $Q$ a finite acyclic quiver, $I$ an
admissible ideal of the path algebra $kQ$ generated by monomial relations and $\m=A\otimes_kkQ/I$. Then $A$ is a left weakly Gorenstein algebra if and only if so is $\m$.  Therefore $\smon(Q,I,\leftidx{^\perp}A)=\leftidx{^\perp}\m$ whenever $A$ or $\m$ is left weakly Gorenstein.
\end{thm}
\endgroup

We also show when $kQ/I$ is not semisimple, the equality $\smon(Q,I,\leftidx{^\perp}A)=\leftidx{^\perp}\m$ holds only if $A$ is left weakly Gorenstein. Together with Theorem \ref{thm C}, this gives another characterization of left weakly Gorensteinness.

 In section 6, we show an example of non-CM-finite left weakly Gorenstein algebra using Theorem \ref{thm C}. We point out that this provides a method to construct new classes of left weakly Gorenstein as well as non-left weakly Gorenstein algebras.

\section{Resolving subcategories and cotorsion pairs}
In this section, we recall some facts about cotorsion pairs. Refer to \cite{AR1,R} for more details.

Let $\Gamma$ be an artin algebra, a full subcategory $\fX$ in $\Gamma$-mod is called {\it resolving} if $\fX$ satisfies (1) closed under extensions and direct summands; (2) contains projective $\Gamma$-modules; (3) closed under kernels of epimorphisms. Dually the {\it coresolving} subcategories are defined.

Denote by $\fX^\perp:=\bigcap\limits_{i=1}^\infty\ker\Ext^i(\fX,-)$ and $\leftidx{^\perp}\!\fY:=\bigcap\limits_{i=1}^\infty\ker\Ext^i(-,\fY)$. By dimension shift, one can observe that:

\begin{lem}\cite[Lemma 3.1]{AR1}\label{res perp_lem}
\begin{enumerate}
\item If $\fX$ is a resolving subcategory in $\Gamma$-mod, then $\fX^\perp=\ker\Ext^1(\fX,-)$ is coresolving.
\item If $\fY$ is a coresolving subcategory in $\Gamma$-mod, then $\leftidx{^\perp}\!\fY=\ker\Ext^1(-,\fY)$ is resolving.
\end{enumerate}
\end{lem}

Let $\ha$ be an additive category and $\mathcal C\subseteq \ha$ a full subcategory. A right $\mathcal C$-approximation of $A\in\ha$ is a morphism $f:C\to A$ such that $C\in\mathcal C$ and $\Hom(\mathcal C,C)\to \Hom(\mathcal C,A)$ is surjective. A subcategory $\mathcal C$ is called a {\it contravariantly finite subcategoy} of $\ha$ if any $A\in\ha$ has a right $\mathcal C$-approximation. The notion of {\it covariantly finite subcategoy} is defined dually. A subcategory is called {\it functorially finite} if it is both contravariantly finite and covariantly finite.
For (co)resolving subcategories, we also have the following generalization of horse-shoe Lemma:

\begin{lem}\cite[Proposition 3.6]{AR1}\label{horseshoe}
Let $\Gamma$ be an artin algebra. Let $0\to A\to B\to C\to 0$ be an exact sequence in $\Gamma$-mod and $\fX$ a resolving subcategory of $\Gamma$-mod. If $f:X\to A$ and $h:Z\to C$ are right minimal $\fX$-approximations, then there is a right $\fX$-approximation $g: Y\to B$ such that the following diagram commutes:
$$
\xymatrix{0\ar[r]&X\ar[r]\ar^{f}[d]&Y\ar[r]\ar^{g}[d]&Z\ar[r]\ar^{h}[d]&0\\
0\ar[r]&A\ar[r]&B\ar[r]&C\ar[r]&0.}
$$
\end{lem}

Let $\mathcal S$ be a set in $\Gamma$-mod.  Denote by  $filt(\mathcal S)$  the full subcategory of
$\Gamma$-mod consisting of module $X$ admitting a filtration $0 = X_0 \subset X_1 \subset ...\subset X_m = X $ of $A$-modules such that
the factors ${X_i}/X_{i-1}$  are all in $\mathcal S$ for all $1 \leq i \leq m$.
Actually, the category $ filt(\mathcal S) $ is the smallest full subcategory of $\Gamma$-mod containing $\mathcal S$  closed under extensions. As an immediate corollary of Theorem \ref{horseshoe}, we have the following:

\begin{cor}\label{simple-app_cor}
Let $\mathcal S$ be a set in $\Gamma$-mod and there is a resolving subcategory $\fX\subseteq filt(\mathcal S)$. Then $\fX$ is a contravariantly finite subcategory of $filt(\mathcal S)$ if and only if each $\Gamma$-module $S\in\mathcal S$ has a right $\fX$-approximation. In particular, $\fX$ is contravariantly finite in $\Gamma$-mod if and only if each simple $\Gamma$-module $S$ has a right $\fX$-approximation.
\end{cor}

Recall that a cotorsion pair $(\fX, \fY)$ is {\it complete} if $\fX$ is contravariantly finite and, $\fY$ is covariantly finite. A cotorsion pair $(\fX, \fY)$ is {\it hereditary} if $\Ext^i(\fX,\fY)=0$ for all $i>0$.

The following results about (complete) hereditary cotorsion pairs are well-known, we include brief proofs for the convenience of reader.
\begin{prop}
Let $\Gamma$ be an artin algebra and $(\fX, \fY)$ a cotorsion pair in $\Gamma$-mod. Then the following are equivalent:
\begin{enumerate}
\item $\fX$ is   resolving,
\item $\fY$ is   coresolving,
\item $(\fX,\fY)$ is hereditary.
\end{enumerate}
\end{prop}
\begin{proof}
Notice that $^\perp \mathcal C$  is resolving and $\mathcal C^\perp$ is coresolving for any subcategory $\mathcal C$.\\
(1)$\implies$(2): Since  $(\fX, \fY)$ is a cotorsion pair, it follows that $\fY=\ker\Ext^1(\fX,-)$.  From Lemma \ref{res perp_lem}, $\fY=\fX^\perp$ is coresolving.
(2)$\implies$(1) is similar.\\
(1),(2)$\implies$ (3): From Lemma \ref{res perp_lem}, $\fY=\fX^\perp$, $\fX=\leftidx^{\perp}\fY$. Hence $\Ext^i(\fX,\fY)=0$ for all $i>0$.\\
(3)$\implies$ (1): Since  $(\fX, \fY)$ is a cotorsion pair, it is clear that $\fX$ contains all projective modules and is closed under extensions and direct summands. It suffice to show $\fX$ is also closed under kernels of epimophisms. Indeed, let $0\to K\to X\to X'\to 0$ be an exact sequence with $X, X'\in \fX$. Since $\Ext^i(\fX,\fY)=0$,  it is easy to see that $K\in\leftidx{^\perp}\fY\subseteq\ker\Ext^1(-,\fY)=\fX$.
\end{proof}

 \begin{prop}
 Let $\Gamma$ be an artin algebra  and $(\fX, \fY)$ a hereditary cotorsion pair in $\Gamma$-mod. Then the following are equivalent:
\begin{enumerate}
\item $\fX$ is a  contravariantly finite subcategory,
\item $\fY$ is a covariantly finite  subcategory,
\item $(\fX,\fY)$ is complete.
\end{enumerate}
 \end{prop}
\begin{proof}
(1) $\iff$ (2) is from \cite{Sa}, see also\cite[Proposition 1.9]{AR1}.

\quad \quad Then (1), (2) $\iff$ (3) by the definition.
\end{proof}

\begin{rem}\label{KS cotorsion}
 In fact, it is proved that in any complete hereditary cotorsion pair $(\fX,\fY)$, both $\fX$ and $\fY$ are functorially finite (\hspace{1sp}\cite[Corollary 2.6]{KS}).
\end{rem}

Homologically finite resolving subcategories are closely related with cotilting modules. A $\Gamma$-module $T$ is {\it cotilting} if (1) $\id T<\infty$; (2) $\Ext^i(T,T)=0$, for all $i>0$; (3) there exists an exact sequence
$
0\to T_n\to T_{n-1}\to\cdots\to T_0\to D\Gamma\to 0,
$
with $T_i\in \add T$.

Following \cite{AR1}, for any subcategory $\fX\subseteq \Gamma$-mod, denote by $$\hat\fX=\{Y | 0\to X_n\to \cdots\to X_1\to X_0\to Y\to 0 \text{\ exact}, X_i\in\fX\}.$$

\begin{thm}\cite[Theorem 5.5]{AR1}\label{cotilting induce}
\begin{enumerate}
\item If $T$ is a cotilting module, then $(\leftidx{^\perp}T,\widehat{\add T})$ is a complete hereditary cotorsion pair.
\item The map $T\to \leftidx{^\perp}T$ is a one-to-one correspondence between basic cotilting modules $T$ and contravariantly finite resolving subcategories $\fX$ of satisfying $\hat \fX=\Gamma$-mod.
\item The map $T\to \widehat{\add T}$ is a one-to-one correspondence between basic cotilting modules $T$ and covariantly finite coresolving subcategories consisting of $\Gamma$-modules with finite injective dimension.
\end{enumerate}
\end{thm}

In summary, we have the following inclusion of sets:
\begin{eqnarray*}
&&\{\text{cotorsion pairs $(\leftidx{^\perp}T,\widehat{\add T})| T$ is a cotilting module }\}\\
& \subseteq& \{\text{complete hereditary cotorsion pairs} (\fX,\fY)\}\\
  &\subseteq& \{\text{hereditary cotorsion pairs}(\fX,\fY)\}.
\end{eqnarray*}
It is possible for both inclusions to be proper.

\section {\bf Separeted monic quiver representations}

\subsection{Separated monic representations}
We fix some notations first. For an acyclic quiver $Q$, we label the vertices of $Q$ as $1,\ 2,\ \cdots,\ n$,
such that if there is an arrow from $j$ to $i$, then $j>i$.

\begin{enumerate}
\item Denote by $\hp(i\to j)$  the set of paths from $i$ to $j$;

\item Denote by $\hp(i\to )$ (resp. $\hp(   \to i)$) the set of paths starting (resp. ending) at $i$;

\item Denote by $\mathcal{A}(i\to j)$ the set of arrows from $i$ to $j$;

\item Denote by $\mathcal{A}(i\to  )$ (resp. $\mathcal{A}(   \to i)$) the set of arrows starting (resp. ending) at $i$.
\end{enumerate}

Let $p$ be a path on $Q$. Denote by $s(p)$ (resp. $e(p)$) the starting (resp. ending) vertex of $p$.
Let $I$ be an admissible ideal generated by monomial relations of $kQ$. For an arrow $\alpha$
we put $$K_\alpha: = \{q\in \hp(\to s(\alpha)) \ | q\not\in I, \ \alpha q \in I\}. \eqno (3.1)$$

Let $\fX$ be an additive full subcategory of $A$-mod. A representation of the bound quiver $(Q,I)$ over $\fX$ is $X =
(X_i, \ X_{\alpha}, \ i\in Q_0, \ \alpha\in Q_1)$, where $X_i\in \fX$ and $X_\alpha: X_{s(\alpha)}\to X_{e(\alpha)}$ are $A$-module homomorphisms such that $X_\gamma:=X_{\alpha_l}\cdots X_{\alpha_1}=0$ for each $\gamma=\alpha_l\cdots\alpha_1$ in a minimal set of generators of $I$. A homomorphism $f=(f_i)_{i\in Q_0}:X\to Y$ between representations consists $f_i\in\Hom_A(X_i,Y_i)$ such that $Y_\alpha f_{s(\alpha)}=f_{e(\alpha)}X_\alpha$ for all $\alpha\in Q_1$.
Denote by $\rep(Q,I,\fX)$ the category of representations of $(Q,I)$ over $\fX$ and simply denote by $\rep(Q,I,A)$ the category of representations of $(Q,I)$ over $A$-mod. It is well-known that there is an equivalence $\rep(Q,I,A)\cong \m$-mod, where $\m=A\otimes kQ/I$. So we will identify  $\m$-modules with representations of $(Q,I)$ over $A$-mod.

For a representation $X$ of  $(Q,I)$ over $\fX$ and a vertex $i\in Q_0$,
the cokernel (resp. kernel ) of the homomorphism
$(X_{\alpha})_{\begin{smallmatrix} \alpha\in Q_1\\ e(\alpha) = i \end{smallmatrix}}:\ \bigoplus\limits_{\begin
{smallmatrix} \alpha\in Q_1\\ e(\alpha) = i \end{smallmatrix}}
X_{s(\alpha)} \longrightarrow X_i$ is denoted by $\cok_i(X)$  (resp. $\Ker_i(X)$). By convention, $\cok_i(X)=X_i$ (resp. $\Ker_i(X)=0$) if $i$ is a source vertex.

In \cite{ZX}, the authors gave the concept of separeted monic representations of quiver over an algebra:

\begin{defn} \label{maindef} \ A representation $X =
(X_i, \ X_{\alpha}, \ i\in Q_0, \ \alpha\in Q_1)$ of the bound
quiver $(Q, I)$ over $\fX$ is a separated monic representation, provided that $X$
satisfies the conditions$:$

\vskip5pt

${\rm (m1)}$  \ For each $i\in Q_0$,  the sum $\sum\limits_{\begin
{smallmatrix} \alpha\in \ha(\to i)
\end{smallmatrix}}\Ima
X_\alpha$ is a direct sum $\bigoplus\limits_{\begin {smallmatrix}
\alpha\in \ha(\to i) \end{smallmatrix}}\Ima X_\alpha;$

\vskip5pt

${\rm (m2)}$  \  For each $\alpha\in Q_1$,   \ $\Ker
X_\alpha =\sum\limits_{q\in K_\alpha} \Ima X_q, $ where $K_\alpha$ is
as in $(3.1)$;

${\rm (m3)}$ \ For each  $i\in Q_0$, $\cok_i(X)\in\fX$.

Denote by ${\rm smon}(Q, I, \fX)$ the category of separated monic rerpesentations of the bound
quiver $(Q, I)$ over $\fX$ and simply by $\smon(Q,I,A)$ the category of separated monic rerpesentations of the bound
quiver $(Q, I)$ over $A$-mod.
\end{defn}

Let $P(i)$ (resp. $I(i)$ or $S(i)$) be the indecomposable projective (resp. injective or simple) $kQ/I$-module at $i\in Q_0$.
It is clear that $P(i)\in {\rm smon}(Q,I, k)$, it follows that
$M\otimes_k P(i)\in {\rm smon}(Q, I, A)$ for $M\in A$-mod. It is clear that $\smon(Q,I,\fX)\subseteq \smon(Q,I,A)$ for any subcategory $\fX\subseteq A$-mod. We will be interested in the case when $\fX$ is closed under extensions.

\begin{lem}\cite[Lemma 2.5]{ZX}\label{ext close_lem}
Let $\fX$ be an additive full subcategory of $A$-mod. Then $\smon(Q, I, \fX )$ is closed under extensions (resp. kernels of epimorphisms; direct summands)
if and only if $\fX$ is closed under extensions (resp. kernels of epimorphisms; direct summands). In particular, $\smon(Q, I, \fX )$ is resolving if and only if $\fX$ is resolving.
\end{lem}

\subsection{Two adjoint pairs}

Notice that $\cok_i(-):\m$-mod$\to A$-mod is a functor. Also denote by $(-)_i: \m$-mod$\to A$-mod the localization functor at branch $i$, i.e. for $X=(X_i, \ X_{\alpha}, \ i\in Q_0, \ \alpha\in Q_1)$, $(X)_i:=X_i$. Recall that both $\cok_i(-)$ and $(-)_i$ fit into adjoint pairs:

\begin{lem}
For each $i\in Q_0$, we have
\begin{enumerate}
\item $(\cok_i(-), -\otimes S(i))$ is an adjoint pair.
\item $(-\otimes P(i), (-)_i)$ is an adjoint pair.
\end{enumerate}
Furthermore $\cok_i(-)$ is exact on $\smon(Q,I,A)$ and $(-)_i$ is an exact functor. $\cok_i(-)$ preserves projective modules and $(-)_i$ preserves both projective and injective modules.
\end{lem}

\begin{proof}
(1) is stated in \cite[Lemma 1.1]{ZX}. We provide a short proof here:

For $X\in\m$-mod, denote by $\pi_i$ the cokernel of $(X_{\alpha})_{\begin{smallmatrix} \alpha\in Q_1\\ e(\alpha) = i \end{smallmatrix}}:\ \bigoplus\limits_{\begin
{smallmatrix} \alpha\in Q_1\\ e(\alpha) = i \end{smallmatrix}}
X_{s(\alpha)} \longrightarrow X_i$. Set

 \begin{eqnarray*}
 \xi^i_{X,Y}:\Hom(\cok_i X,Y)&\to& \Hom(X, Y\otimes S(i))\\
 g&\mapsto& (f_j)_{j\in Q_0}
 \end{eqnarray*}
where $f_i=g\circ \pi_i$ and $f_j= 0$ for $j\neq i$.

 Notice that for $X\in\m$-mod and $Y\in A$-mod, a morphism $(f_j)_{j\in Q_0}\in\Hom (X,Y\otimes S(i))$ if and only if $f_i\circ X_\alpha=0$ for all $\alpha\in Q_1$ with $e(\alpha)=i$ and $f_j= 0$ for $j\neq i$, if and only if $f_i\circ (X_{\alpha})_{\begin{smallmatrix} \alpha\in Q_1\\ e(\alpha) = i \end{smallmatrix}}=0$ and $f_j= 0$ for $j\neq i$. By the universal property of cokernels, this is equivalent to saying there exists a unique $g:\cok_i(X)\to Y$ such that  $f_i=g\circ \pi_i$. Hence $\xi^i_{X,Y}$ is a bijection. The naturality of $\xi^i_{X,Y}$ is straightforward to check.

 For the proof of (2), we refer to \cite[Lemma 1.2]{LZ3}.

The exactness of $\cok_i(-)$ on $\smon(Q,I,A)$ is proved in \cite[Lemma 2.5]{ZX} and the remaining assertions are straightforward.
\end{proof}

We need a general result about adjoint pairs:

\begin{lem}\label{adj pair_lem}
Let $A$, $B$ be artin algebras and $F:A$-$mod\to B$-$mod$ a functor left adjoint to $G:B$-$mod\to A$-$mod$.
\begin{enumerate}
\item If $F$ is exact on a resolving subcategory $\fX\subseteq A$-mod and preserves projective objects, then for $M\in\fX$ and $N\in B$-mod, $\Ext_B^k(FM,N)\cong \Ext^k_A(M,GN)$, $\forall k\geq 0$.
\item If $G$ is exact on a coresolving subcategory $\fY\subseteq B$-mod and preserves injective objects, then for $M\in A$-mod and $N\in\fY$, $\Ext_B^k(FM,N)\cong \Ext^k_A(M,GN)$, $\forall k\geq 0$.
\end{enumerate}

\end{lem}
\begin{proof}
We just prove for (1).
Take a projective resolution of the $A$-module $M\in\fX$:
$$
\cdots \to P_1\to P_0\to M\to 0.
$$
Since $\fX$ is resloving, each syzygy $\Omega^k M\in\fX$, $k\geq 0$.  By the hypothesis that $F$ is exact on $\fX$ and preserves projective objects, we obtain a projective resolution of $B$-modules:
$$
\cdots \to FP_1\to FP_0\to FM\to 0.
$$
Applying $\Hom_B(-,N)$ to it and using the fact that $(F,G)$ is an adjoint pair, we have the following isomorphism between complexes:
$$
\xymatrix{0\ar[r]& \Hom_B(FM,N)\ar[r] \ar[d]^{\wr}&\Hom_B(FP_0,N)\ar[r]\ar[d]^{\wr} &\Hom_B(FP_1,N)\ar[r]\ar[d]^{\wr}&\cdots \\
0\ar[r]& \Hom_A(M,GN)\ar[r] &\Hom_A(P_0,GN)\ar[r] &\Hom_A( P_1,GN)\ar[r]&\cdots}
$$
 Therefore, we obtain isomorphisms between the homology groups: $$\Ext_B^k(FM,N)\cong \Ext^k_A(M,GN),\forall k\geq 0. $$
 \end{proof}

As an immediate application, we have the following key lemma:
\begin{lem}\label{adj ext_lem} For each $k\geq 0$, there are isomorphisms:
\begin{enumerate}
\item  $\Ext_A^k(\cok_i(Y),X)\cong \Ext^k_\m(Y,X\otimes S(i))$, for $X\in \m$-mod and $Y\in\smon(Q,I,A)$.
\item  $\Ext^k_\m(X\otimes P(i),Y)\cong \Ext_A^k(X,Y_i)$,  for $X\in A$-mod and $Y\in\m$-mod.
\end{enumerate}
\end{lem}

\subsection{The filtration interpretations}


Recall that for a subset $\mathcal S\subseteq A$-mod, the filtration category $filt(\mathcal S)$ is the smallest extension-closed subcategory which contains $\mathcal S$. From the definition, it is clear that
if $\fY$ is an extension-closed subcategory of $A$-mod, then $$\rep(Q,I,\fY)=filt(\fY\otimes \mathcal S(kQ/I)),$$ where $\fY\otimes \mathcal S(kQ/I)=\{Y\otimes S| Y\in\fY, S \text{ is a simple $kQ/I$-module} \}$.

\vskip10pt

Furthermore, there is also a filtration interpretation of the category $\smon(Q,I,\fX)$.

\begin{thm}\cite[Theorem 4.1]{ZX}\label{filt_thm}
 Let $\fX$ be an extension-closed subcategory of $A$-mod. Then
$$\smon(Q, I, \fX ) = filt(\fX\otimes \Proj(kQ/I)),$$
where $\fX\otimes \Proj(kQ/I)=\{X\otimes P| X\in\fX, P \text{ is a projective $kQ/I$-module} \}$.
\end{thm}

\subsection{The Cartan-Eilenberg isomorphism}
We will heavily use the the following general result for tensor products of finite dimensional algebras, which is often referred to as the Cartan-Eilenberg isomorphism:
\begin{thm}\cite[Theorem 3.1, p.209, p.205]{CE}\label{CE iso}
Let $A,B$ be finite dimensional algebras over $k$, $\otimes=\otimes_k$. Let $L,M\in A$-{\rm mod} and $U,V\in B$-{\rm mod}. Then there is an isomorphism
$$\Ext^m_{
A\otimes B}(L \otimes U, M\otimes V )\cong
\sum_{p+q=m}
(\Ext^p_
A(L, M) \otimes \Ext^q_
B(U, V )),\forall m \geq 0.$$
\end{thm}

We mention a fact about the projective dimensions of tensor products for later applications.
\begin{cor}[to Thoerem \ref{CE iso}]\label{CE cor}
Let $L$ be an $A$-module and $U$ be a $B$-module. Then $$\pd_{\m}L\otimes_k U=\pd_A L+\pd_{B} U.$$
\end{cor}
\begin{proof}
 Assume $\pd_A L=s$ and $\pd_B U=t$. On one hand, by K\"unneth formula \cite{We}, $L\otimes U$ has a projective resolution of length $s+t$. Hence $\pd_\m L\otimes U\leq s+t$.
   On the other hand, since there are modules $S$ and $T$ such that $\Ext^s_A(L,S)\neq 0$ and $\Ext^t_B(U,T)\neq0$,
  $\Ext^{s+t}_{\m}(L \otimes U, S\otimes T )$ has a summand $\Ext^s_A(L,S)\otimes\Ext^t_B(U,T)\neq 0$. Hence $\pd_{\m} L\otimes U\geq s+t$.
\end{proof}

\section {\bf The separated monic  correspondence}
 Let $A$ be a finite dimensional algebra, $Q$ a finite acyclic quiver, $I$ an
admissible ideal of the path algebra $kQ$ generated by monomial relations and $\m=A\otimes_kkQ/I$. In this section, we are going to investigate the map $\fX\mapsto \smon(Q,I,\fX)$ for various resolving subcategories $\fX$.  We start from a brief recollection for a reciprocity of separated monic representations studied in \cite{Z,ZX}.

\subsection{A Reciprocity of separated monic representations}

\begin{thm}\cite[Theorem 2.6]{ZX}\label{smon corr_thm}
If $T$ is a cotilting $A$-module, then $${\rm smon}(Q,I, \leftidx{^\perp}T)=\leftidx{^\perp}(T\otimes kQ/I).$$
\end{thm}

\begin{rem}
Since $Q$ is an acyclic quiver, a $\m$-module $T\otimes kQ/I$ is cotilting if and only if $T$ is a cotilting module. Therefore, the category $\mathcal T=\leftidx{^\perp}(T\otimes kQ/I)$ above satisfies $\widehat{\mathcal T}=\m$-mod (see Theorem \ref{cotilting induce} above).
\end{rem}

Taking $T=DA$ and combine Remark \ref{KS cotorsion}, one can easily see the following:
\begin{cor}\label{smon corr_cor}
The subcategory ${\rm smon}(Q,I, A)=\leftidx{^\perp}(DA\otimes kQ/I)$ is functorially finite.
\end{cor}

\subsection{Hereditary cotorsion pairs}
In this subsection, we will investigate the separated monic correspondence for hereditary cotorsion pairs.
 \begin{prop}\label{cotorsion prop 1}
If $(\fX,\fY)$ is a hereditary cotorsion pair, then $({\rm smon}(Q, I, \fX ), {\rm rep}(Q, I, \fY))$ is a hereditary cotorsion pair.
\end{prop}
\begin{proof}
For the sake of convenience, denote by $\widetilde\fX={\rm smon}(Q, I, \fX )$ and $\widetilde \fY={\rm rep}(Q, I, \fY)$. Notice that $(\fX,\fY)$ being a hereditary cotorsion pair means $\fX^\perp=\fY$ and $\leftidx{^\perp}\fY=\fX$ (Lemma \ref{res perp_lem}). We prove the statement as the following procedures:

\vskip5pt

\noindent {\bf Claim 1:} ${\widetilde\fX}\subseteq \leftidx{^\perp}{\widetilde\fY} $ and ${\widetilde\fY}\subseteq {\widetilde\fX}^\perp$.

From Theorem \ref{CE iso}, for any $X\otimes P(i)\in\widetilde\fX$ and $Y\otimes S(j)\in\widetilde \fY$,
$$\Ext^k(X\otimes P(i), Y\otimes S(j))\cong \Ext^k(X , Y)\otimes \Hom(P(i),S(j))=0, \text{for\ } k>0.$$
Since $\widetilde\fX= filt(\fX\otimes \hp (kQ/I))$ and $\widetilde \fY=filt(\fY\otimes \mathcal S(kQ/I))$,  it follows that $\Ext^k(\widetilde \fX, \widetilde \fY) =0, \text{for\ } k>0,$ which proves Claim 1.

\vskip5pt

\noindent {\bf Claim 2:} $\leftidx{^\perp}{\widetilde\fY}\subseteq {\widetilde\fX}$.

Let $X\in \leftidx{^\perp}{\widetilde\fY}$. First, since $\fY$ contains injective $A$-modules, $\leftidx{^\perp}{\widetilde\fY}\subseteq \leftidx{^\perp}(DA\otimes kQ/I)={\rm smon}(Q, I, A )$, where the second equality follows from Corollary \ref{smon corr_cor}. So $X\in {\rm smon}(Q, I, A )$.

Second, by Lemma \ref{adj ext_lem} (1), for any $Y\in\fY$, there is an isomorphism $$\Ext^k_A(\cok_i (X), Y )\cong\Ext^k_\m(X, Y\otimes S(i) )=0, k>0,$$ which implies that
$\cok_i(X)\in \leftidx{^\perp}\fY=\fX$.
Therefore $X\in\smon(Q,I,\fX)=\widetilde{\fX}$.

 \vskip5pt

\noindent {\bf Claim 3:} $ {\widetilde\fX}^\perp \subseteq{\widetilde\fY}$.

Let $Y\in {\widetilde\fX}^\perp$.
  Since $X\otimes P(i)\in\widetilde\fX$ for any $X\in\fX$, it follows from Lemma \ref{adj ext_lem} (2) that  $$\Ext_A^k(X,Y_i)\cong \Ext_\m^k(X\otimes P(i),Y)=0,  k>0.$$

  Therefore $Y_i\in\fX^\perp=\fY$ and hence $Y\in\rep(Q,I,\fY)=\widetilde\fY$.

 \vskip5pt

So far, we have proved $ {\widetilde\fX}=\leftidx{^\perp}{\widetilde\fY}$ and $ {\widetilde\fY}={\widetilde\fX}^\perp $. Finally, we can finish the proof as below. According to Lemma \ref{ext close_lem}, $\fX$ is resolving implies so is $\widetilde\fX$. Hence by Lemma \ref{res perp_lem}, $\widetilde\fY={\widetilde\fX}^\perp=\ker \Ext^1(\widetilde\fX, -)$ is a coresolving subcategory and  $\widetilde\fX=\leftidx{^\perp}{\widetilde\fY}=\ker \Ext^1(-,\widetilde\fY )$. Therefore  $(\widetilde\fX, \widetilde\fY)$ is a hereditary cotorsion pair.
 \end{proof}

Next, we will show the converse of Proposition \ref{cotorsion prop 1}.

\begin{prop}\label{cotorsion prop 2}
If ${\rm smon}(Q, I, \fX )$ is a resolving cotorsion class for some subcategories $\fX$ in $A$-mod, then $\fX$ is a resolving cotorsion class.
\end{prop}

\begin{rem}\label{cotorsion 2_rem}
We don't need to assume ${\rm smon}(Q, I, \fX )^\perp$ is of the form $\rep(Q,I,\fY)$. Nevertheless, under the condition that ${\rm smon}(Q, I, \fX )$ is a resolving cotorsion class, the validity of Proposition \ref{cotorsion prop 2} together with Proposition \ref{cotorsion prop 1} would suggest that ${\rm smon}(Q, I, \fX )^\perp=\rep(Q,I,\fX^\perp)$.
\end{rem}

Before proving Proposition \ref{cotorsion prop 2}, we need the following lemma.

\begin{lem}\label{cotorsion lem 3}
 If $Y\in{\rm smon}(Q, I, \fX )^\perp$, then the each branch $Y_i\in\fX^\perp$.
\end{lem}
\begin{proof}
Due to Lemma \ref{adj ext_lem}, $\Ext_A^k(\fX , Y_i)=\Ext_\m^k(\fX\otimes P(i), Y)=0$, for $k>0$. So $Y_i\in\fX^\perp$.
\end{proof}

{\bf Proof of Proposition \ref{cotorsion prop 2}}
For convenience, denote by $\widetilde \fX={\rm smon}(Q, I, \fX )$. From Lemma \ref{ext close_lem}, $\widetilde\fX$ is resolving implies that $\fX$ is resolving. So to show $\fX$ is a cotorsion class, it suffices to prove that $\fX=\leftidx{^\perp}(\fX^\perp)$.

Obviously, $\fX\subseteq\leftidx{^\perp}(\fX^\perp)$. Now let $M\in\leftidx{^\perp}(\fX^\perp)$ and $1$ be a sink vertex in $Q$. We claim that $M\otimes P(1)\in\widetilde\fX$. In fact, by Lemma \ref{adj ext_lem}, for any $Y\in\widetilde\fX^\perp$, $\Ext^k_\m(M\otimes P(1) , Y)=\Ext_A^k(M, Y_1)=0$, for $k>0$, where the second equality follows from Lemma \ref{cotorsion lem 3}. Hence $M\otimes P(1)\in \leftidx{^\perp}(\widetilde\fX^\perp)=\widetilde\fX$.
Because $1$ is a sink vertex, $M=\cok_1 (M\otimes P(1))\in\fX$. Hence $ \leftidx{^\perp}(\fX^\perp)\subseteq \fX$.
$\hfill\qed$

\vskip10pt

\subsection{Complete hereditary cotorsion pairs}
The second goal of this section is to study the homological finiteness of hereditary cotorsion pairs under the separated monic correspondence.

We need the following general result about contravariantly finite subcategories:
\begin{lem}\label{contra transitive_lem}
Let $\mathcal A$ be an additive category and $\mathcal C\subset\mathcal D$ be additive full subcategories of $\mathcal A$. If $\mathcal D$ is a contravariantly finite subcategory of $\mathcal A$ and $\mathcal C$ is a contravariantly finite subcategory of $\mathcal D$, then $\mathcal C$ is a contravariantly finite subcategory of $\mathcal A$.
\end{lem}
\begin{proof}
For any $M\in\mathcal A$. Let $f:D\to M$ be a right $\mathcal D$-approximation of $M$ and $g:C\to D$ be a right $\mathcal C$-approximation of $D$. Then $f\circ g$ is a right $\mathcal C$-approximation of $M$.
\end{proof}

\begin{lem}\label{contra subcat_lem}
If $\fX$ is a contravariantly finite resolving subcategory of $A$-mod, then $\smon(Q,I,\fX)$ is a contravariantly finite subcategory of $\smon(Q,I,A)$.
\end{lem}
\begin{proof}
From Theorem \ref{filt_thm}, $\smon(Q,I,A)=filt(A$-mod$\otimes \Proj(kQ/I))$. In the light of Corollary~\ref{simple-app_cor}, it suffices to show each $\m$-module $M\otimes P(i)$ has a right $\smon(Q,I,\fX)$-approximation.

In fact, let $f:X\to M$ be a minimal right $\fX$-approximation of the $A$-module $M$, then there is an exact sequence of $A$-modules:
$$
0\to Y\to X\stackrel{f}\to M\to 0,
$$
where $f$ is an epimorphism because $\fX$ contains all projective $A$-modules. It also follows that $\Ext^1(\fX,Y)=0$ by  Wakamastu's Lemma (see \cite{W}, \cite[Lemma 1.3]{AR1}) and hence $Y\in\fX^\perp$.
Applying $-\otimes P(i)$, we obtain an exact sequence of $\m$-modules:
$$
0\to Y\otimes P(i)\to X\otimes P(i)\stackrel{f\otimes 1}\to M\otimes P(i)\to 0,
$$
where $X\otimes P(i)\in \smon(Q,I,\fX)$ and $Y\otimes P(i)\in\rep(Q,I,\fX^\perp)=\smon(Q,I,\fX)^\perp$ by Proposition \ref{cotorsion prop 1}. Hence $f\otimes 1$ is a right $\smon(Q,I,\fX)$-approximation of $M\otimes P(i)$.
\end{proof}

\begin{prop}\label{cotorsion prop 3}
If $(\fX,\fY)$ is a complete hereditary cotorsion pair, then the cotorsion pair $({\rm smon}(Q, I, \fX ), {\rm rep}(Q, I, \fY))$ is also complete hereditary.
\end{prop}
\begin{proof}
By Corollary \ref{smon corr_cor}, $\smon(Q,I,A)$ is a contravariantly finite subcategory of $\m$-mod. By Lemma \ref{contra subcat_lem}, $\smon(Q,I,\fX)$ is a contravariantly finite subcategory of $\smon(Q,I,A)$. Therefore due to Lemma \ref{contra transitive_lem}, $\smon(Q,I,\fX)$ is a contravariantly finite subcategory of $\m$-mod. It follows that $({\rm smon}(Q, I, \fX ), {\rm rep}(Q, I, \fY))$ is a complete hereditary cotorsion pair.
\end{proof}

Next, we will show the converse of Proposition \ref{cotorsion prop 3}.
\begin{prop}\label{cotorsion prop 4}
If ${\rm smon}(Q, I, \fX )$ is a contravariantly finite resolving subcategory for some  subcategory $\fX$ in $A$-mod, then $\fX$ is a contravariantly finite resolving subcategory.
\end{prop}

\begin{proof}
For any $A$-module $M$, consider the minimal right ${\rm smon}(Q, I, \fX )$-approximation $f:X\to M\otimes P(i)$ of the $\m$-module $M\otimes P(i)$, $\forall i\in Q_0$. There is an exact sequence of $\m$-modules:
$$
0\to Y\to X\stackrel{f}\to M\otimes P(i)\to 0,
$$
where  $Y\in\smon(Q,I,\fX)^\perp$ from Wakamastu's Lemma, and hence $Y\in\rep(Q,I,\fX^\perp)$ due to Proposition \ref{cotorsion prop 2} and Remark \ref{cotorsion 2_rem}.

Applying the localization functor $(-)_i$, we obtain and exact sequence of $A$-modules:
$$
0\to Y_i\to X_i\stackrel{f_i}\to M\to 0,
$$
where $X_i\in\fX$ and $Y_i\in\fX^\perp$. Hence $f_i$ is a right $\fX$-approximation of $M$. It follows that $\fX$ is a contravariantly finite subcategory and it is resolving from Lemma \ref{ext close_lem}.
\end{proof}

Now Theorem \ref{thm A} follows immediately from Proposition \ref{cotorsion prop 1}, Proposition \ref{cotorsion prop 2}, Proposition \ref{cotorsion prop 3} and Proposition \ref{cotorsion prop 4}.

\begin{rem}\label{thm A rem}
(1) There is a complete hereditary cotorsion pair $(\smon(Q,I,A),\rep(Q,I,\add(DA)))$.
(2) It is worth mentioning that {\bf we don't known} if any contravariantly finite resolving subcategory of $\m$-mod which is contained in $\smon(Q,I,A)$ has to be of the form $\smon(Q,I,\fX)$ for some subcategory $\fX$.
\end{rem}

\subsection{A dual version} We state the dual version of Theorem \ref{thm A} using the notion of separated epic representations, for later applications.

\begin{defn}
A representation
$X = (X_i, X_\alpha)\in \rep(Q, I, A)$ is {\it separated epic}, if $X$ satisfies the following conditions:\\
$\rm (e1)$ For $i\in Q_0$, $\Ima(X_i\stackrel{(X_\alpha)_{\alpha\in\ha(i\to)}}\longrightarrow
 \bigoplus\limits_{\alpha\in \ha(i\to)}X_{e(\alpha)})=\bigoplus\limits_{\alpha\in\ha(i\to)}\Ima X_\alpha$;\\
$\rm (e2)$ For $\alpha\in Q_1$, $\Ima X_\alpha =\bigcap\limits_{q\in L_\alpha}\ker X_q$, where
$$L_\alpha:= \{ \text{ non-zero path  $q$  of length } \geq 1 | s(q) =e(\alpha), q\alpha\in  I\};$$
$\rm (e3)$ For $i\in Q_0$, $\ker_i(X) := \bigcap\limits_{\alpha\in\ha(i\to)}\ker X_\alpha\in\fX$.
\end{defn}

For a subcategory $\fX\subseteq A$-mod, $D=\Hom_k(-,k)$  yields dualities (\hspace{1sp}\cite[Proposition 6.1]{ZX}): $$\sepi(Q, I, \fX ) = D\smon(Q^{\op}, I^{\op}, D\fX).$$
$$\rep(Q, I, \fY ) = D\rep(Q^{\op}, I^{\op}, D\fY).$$

Hence, we have the following dual version of Theorem \ref{thm A}:

 \newcounter{temp}
\begingroup
\setcounter{temp}{1}
\renewcommand\thethm{\Alph{temp}$'$}

\begin{thm}\label{thm A'}
Let $A$ be
a finite dimensional algebra, $Q$ a finite acyclic quiver, $I$ an
admissible ideal of the path algebra $kQ$ generated by monomial relations and $\m=A\otimes_kkQ/I$.  Then
$(\fX,\fY)$ is a hereditary cotorsion pair in $A$-mod if and only if $({\rm rep}(Q, I,\fX ),{\rm sepi}(Q,I,\fY))$ is a hereditary cotorsion pair in $\m$-mod.
Furthermore, $(\fX,\fY)$ is a complete hereditary cotorsion pair if and only if so is $({\rm rep}(Q, I,\fX ),{\rm sepi}(Q,I,\fY))$.
\end{thm}
\endgroup
\addtocounter{thm}{-1}

Applying Remark \ref{KS cotorsion} to Theorem \ref{thm A} and Theorem \ref{thm A'}, we have the conclusions below:

\begin{cor}
 (1) If $\fX$ is a contravariantly finite resolving subcategory of $A$-mod, then the subcategories $\smon(Q,I,\fX)$ and $\rep(Q,I,\fX)$ are functorially finite subcategories of $\m$-mod and hence have Auslander-Reiten sequences.\\
 (2) If $\fY$ is a covariantly finite coresolving subcategory of $A$-mod, then the subcategories $\sepi(Q,I,\fY)$ and $\rep(Q,I,\fY)$ are functorially finite subcategories of $\m$-mod and hence have Auslander-Reiten sequences.
\end{cor}

\section {\bf Separated monic representations of semi-Gorenstein-projective modules}

\subsection{Semi-Gorenstein-projective modules}
Following \cite{RZ2}, for a finite dimensional $k$-algebra $A$, modules in $\leftidx{^\perp}A$ are called {\it semi-Gorenstein-projective modules}. 
Recall that $\m=A\otimes kQ/I$, where $Q$ is an acyclic quiver and $I$ a monomial ideal. Due to Theorem \ref{LZ3 thm}, Gorenstein-projective $\m$-modules are exactly separated monic representations of $(Q,I)$ over $\Gp(A)$, namely $\Gp(\m)=\smon(Q,I,\Gp(A))$. Applying the separated monic correspondence (Theorem \ref{thm A}) to the cotorsion class $\leftidx{^\perp}A $, we obtain a cotorsion class $\smon(Q,I, \leftidx{^\perp}A)$ in $\m$-mod.  This section is devoted to provide a complete answer to the following natural question.
\begin{que}\label{que}
 Whether the subcategory   $\leftidx{^\perp}\m$  of semi-Gorenstein-projective $\m$-modules coincides with $\smon(Q,I,\leftidx{^\perp}A)$?
\end{que}
The aim of this section is to give a full answer to this question. It turns out that this question is closely related with the left weakly Gorensteinness of the algebra $\m$.

 \subsection{Left weakly Gorenstein algebras}
Recall that a finite dimensional algebra $A$ is called {\it left weakly Gorenstein}, if $\Gp(A)=\leftidx{^\perp}{A}$ as subcategories of left $A$-modules. 

Recall that for artin algebras $A$, $B$ and an $B$-$A$-bimodule $M$, the upper triangular matrix ring is $\m=\begin{bmatrix}B&M\\0&A\end{bmatrix}$ with addition and multiplication given by the ones of matrices. A left $\m$-module is identified with a triple $\begin{bmatrix} X\\Y\end{bmatrix}_\phi$, where $X\in B$-mod, $Y\in A$-mod and $\phi: M\otimes_A Y\to X$ is a $B$-module homomorphism (see \cite{FGR} for details).

Specializing Theorem \ref{matrixring_thm} for the situation when both algebras $A$ and $B$ are left weakly Gorenstein, we obtain the following lemma:

\begin{lem}\label{matrixring lwg_lem}
Let $\m=\begin{bmatrix}B&_BM_A\\0&A\end{bmatrix}$ be an artin algebra, where $A$ and $B$ are both left weakly Gorenstein. Assume that $\pd _BM<\infty$ and $D(M_A)\in(\leftidx{^\perp}A)^\perp$.
Then a $\m$-module $\begin{bmatrix}X\\Y\end{bmatrix}_\phi\in\leftidx{^\perp}\m$ if and only if $\phi$  is a monomorphism, $\cok \phi\in\Gp(B)$ and $Y\in\Gp(A)$.
\end{lem}

\begin{proof}
``if part'': Assume that $0\to M\otimes Y\stackrel{\phi}\to X\to \cok\phi\to0$ is exact and $\cok\phi\in\Gp(B)$. Applying $\Hom(-,B)$ to this exact sequence, it is easy to see that $\phi^*$ is an epimorphism and $\Ext^k(\phi,B)$ are isomorprhisms for all $k>0$. Furthermore, $Y\in\Gp(A)\subseteq \leftidx{^\perp}A$. Hence, by the ``if part'' of Theorem \ref{matrixring_thm}, $\begin{bmatrix}X\\Y\end{bmatrix}_\phi\in\leftidx{^\perp}\m$.

``only if part'': Let $0\to L\to P\stackrel{\pi}\to X\to 0$ an exact sequence, where $\pi$ is the projective cover. Taking its pull-back with the morphism $\phi$, we have a commutative diagram with exact rows:
$$\xymatrix{0\ar[r]&L\ar[r]\ar@{=}[d]&E\ar[r]\ar[d]^{\psi\ \ \ (p.b.)}&M\otimes Y\ar[r]\ar[d]^\phi&0.\\
0\ar[r]&L\ar[r]&P\ar[r]^\pi&X\ar[r]&0 }$$

Applying $\Hom_B(-,B)$ to the first two rows, we obtain a commutative diagram with long exact sequences, starting with
$$\xymatrix{0\ar[r]&X^*\ar[r]\ar@{->>}[d]^{\phi^*}&P^*\ar[r]\ar[d]^{\psi^*}&L^*\ar[r]\ar@{=}[d]& \Ext^1_B(X, B)\ar[r]\ar[d]_\wr^{\Ext_B^1(\phi,B)}&\cdots \\
0\ar[r]&(M\otimes_A Y)^*\ar[r]&E^*\ar[r]&L^*\ar[r]&\Ext_B^1(M\otimes Y,B)\ar[r]&\cdots }$$
and continues for any $k>0$ as:
$$\xymatrix{\cdots\ar[r]&0\ar[r]\ar[d]&\Ext^k(L,B)\ar[r]\ar@{=}[d]& \Ext^{k+1}(X,B)\ar[r]\ar[d]_\wr^{\Ext^{k+1}(\phi,B)}& 0\ar[r]\ar[d]&\cdots \\
\cdots\ar[r]& \Ext^k(E,B)\ar[r]&\Ext^k(L,B)\ar[r]& \Ext^{k+1}(M\otimes Y,B)\ar[r]& \Ext^{k+1}(E,B)\ar[r]&\cdots },$$
where $\phi^*$ is an epimorphism and $\Ext^k_B(\phi,B)$ are isomorphisms for all $k>0$ due to the ``only if part'' of Theorem \ref{matrixring_thm}.

Hence, it is easy to see that $E\in\leftidx{^\perp}B=\Gp(B)$ and $\psi^*$ is an epimorphism. As $\Hom_B(-,B)$ is a duality on the subcategory $\Gp(B)$, it follows that $\psi$ is a monomorphism. Thus, so is $\phi$ due to the pull-back diagram.

On the other hand, since $\psi^*$ is an epimorphism, $\psi$ is a left $\Proj(B)$-approximation of $E$. Hence $\cok\psi\cong\cok\phi$ is again a Gorenstein-projective $B$-module.
\end{proof}

This lemma characterizes semi-Gorenstein projective modules over $\m$. On the other hand, it is known that there is a similar characterization of Gorenstein-projective modules over matrix rings:

\begin{thm}\cite[Theorem 1.4]{Z2}\label{Z2 thm}
Let $A$, $B$ be artin algebras and $\m=\begin{bmatrix}B&_BM_A\\0&A\end{bmatrix}$. Assume $M$ is a compatible $B$-$A$-bimodule. i.e. $M$ satisfies the following conditions:
\begin{enumerate}
\item For any exact sequence $Q^\bullet$ of projective $A$-modules, $M\otimes Q^\bullet$ is exact.
\item For any exact sequence $P^\bullet$ of projective $B$-modules, $\Hom(P^\bullet, M)$ is exact. 	
\end{enumerate}
 Then $\begin{bmatrix}X\\Y\end{bmatrix}_\phi\in \Gp(\m)$ if and only if $\phi$  is a monomorphism, $\cok \phi\in\Gp(B)$ and $Y\in\Gp(A)$.
\end{thm}

Combining Lemma \ref{matrixring lwg_lem} and Theorem \ref{Z2 thm}, we can show the following result about the left weakly Gorensteinness of matrix rings, which improves Theorem \ref{Z3 thm}.

\begin{prop}\label{matrixring lwg_prop}
Let $A$, $B$ be artin algebras and $\m=\begin{bmatrix}B&_BM_A\\0&A\end{bmatrix}$.  Assume that $\pd _BM<\infty$, $\pd M_A<\infty$ and $D(M_A)\in(\leftidx{^\perp}A)^\perp$. Then $\m$ is left weakly Gorenstein if and only if both $A$ and $B$ are left weakly Gorenstein.
\end{prop}

\begin{proof}
Since $\pd _BM<\infty$ and $\pd M_A<\infty$, $M$ is a compatible $B$-$A$-bimodule, according to \cite[Proposition 1.3(1)]{Z2}.

If $\m$ is a left weakly Gorenstein algebra, then both $A$ and $B$ are left weakly Gorenstein according to Theorem \ref{Z3 thm}.

Conversely, assume both $A$ and $B$ are left weakly Gorenstein. Let $\begin{bmatrix}X\\Y\end{bmatrix}_\phi\in\leftidx{^\perp}\m$. According to Lemma \ref{matrixring lwg_lem}, $\phi$  is a monomorphism, $\cok \phi\in\Gp(B)$ and $Y\in\Gp(A)$.
Hence by Theorem \ref{Z2 thm}, $\begin{bmatrix}X\\Y\end{bmatrix}_\phi\in \Gp(\m)$. Therefore, $\m$ is left weakly Gorenstein.
\end{proof}

\begin{constr}\label{matrix ring_constr}
Now we apply this result to the algebra $\m=A\otimes kQ/I$. Let $n$ be a source vertex of the quiver $Q$. Then $\m'=\m\otimes(1-e_n)(kQ/I)(1-e_n)=A\otimes kQ'/I'$ is a subalgebra of $\m$. Here $Q'$ is again an acyclic quiver, obtained by deleting vertex $n$ of $Q$ and $I'$ is again a monomial ideal. The algebra $\m=A\otimes kQ/I$ can also be viewed as an upper triangular matrix ring: $$\m=\begin{bmatrix}\m' & A\otimes \rad P(n)\\ 0 &A \end{bmatrix}.$$ 
\end{constr}

Now we show the main theorem of this section, which implies that Proposition 4.8 in \cite{Z3} still holds for algebras $\m=A\otimes kQ/I$, where $Q$ is acyclic and $I$ is a monomial ideal.

\begin{thm}\label{lwg smon_lem}
The algebra $A$ is  left weakly Gorenstein if and only if so is $\m=A\otimes kQ/I$.
\end{thm}
\begin{proof}
We prove the first statement using induction on $|Q_0|$. If $|Q_0|=1$, the statement follows trivially. Now assume $|Q_0|=n>1$.
Let $\m=\begin{bmatrix}\m' & A\otimes \rad P(n)\\ 0 &A \end{bmatrix}$ for some source vertex $n\in Q_0$ as in Construction \ref{matrix ring_constr}.

First, we check that $\m$, as a triangular matrix ring, satisfies the hypotheses in Proposition \ref{matrixring lwg_prop}.
According to  Corollary \ref{CE cor}, $$\pd_{\m'} A\otimes \rad P(n)=\pd_AA+\pd_{kQ'/I'} \rad P(n)<\infty.$$
As a right $A$-module, $A\otimes \rad P(n)$ is isomorphic to copies of $A$, thus $\pd A\otimes \rad P(n)<\infty$. Also because $D(A\otimes \rad P(n))_A$ is injective,  it is in $(\leftidx{^\perp}A)^\perp$.

Now we can finish the proof as the following. If $A$ is left weakly Gorenstein, so is $\m'=A\otimes kQ'/I'$ by the induction hypothesis. Hence according to Proposition \ref{matrixring lwg_prop}, $\m$ is left weakly Gorenstein. Conversely, if $\m$ is left weakly Gorenstein,   $A$ is left weakly Gorenstein by Proposition \ref{matrixring lwg_prop}.
 \end{proof}

We have  an immediate corollary below,  which completes the proof of Theorem \ref{thm C}.

\begin{cor}
When either $A$ or $\m$ is left weakly Gorenstein, $\smon(Q,I,\leftidx{^\perp}A)=\leftidx{^\perp}\m$.
\end{cor}

\begin{proof}
By Theorem \ref{lwg smon_lem}, if either $A$ or $\m$ is left weakly Gorenstein, both of them are left weakly Gorenstein. Hence it follows that
$$
\smon(Q,I,\leftidx{^\perp}A)=\smon(Q,I,\Gp(A))=\Gp(\m)=\leftidx{^\perp}\m,
$$
where the second equality is due to Theorem \ref{LZ3 thm}.
\end{proof}

We show the converse of this corollary is also true whenever $Q$ contains at least one arrow, which yields another characterization of left weakly Gorensteinness.

\begin{prop}
As long as $kQ/I$ is not semisimple, the equality $\smon(Q,I,\leftidx{^\perp}A)=\leftidx{^\perp}\m$ holds only if $A$ is left weakly Gorenstein.
\end{prop}
\begin{proof}
If $A$ is not a left weakly Gorenstein algebra, then there is an indecomposable semi-Gorenstein-projective $A$-module $U$, which is not torsionless  (see \cite[Theorem 1.2]{RZ2}). Hence a left $\Proj(A)$-approximation $\phi: U\to P $ is not a monomorphism.   According to Theorem \ref{matrixring_thm}, for any $r>0$, $X=\begin{bmatrix} P^r\\U\end{bmatrix}_{\phi^r}$ is a semi-Gorenstein-projective module over the matrix ring $T_2^r(A)=\begin{bmatrix}A&A^r\\0&A\end{bmatrix}$.

Now since $kQ/I$ is not semisimple, $Q$ contains a subquiver
$\xymatrix{
n \ar@<1.2ex>[r]_{\cdots\ \ \ } \ar @<-1.2ex>[r]&n-1 }
$,   where $n$ is a source vertex and there are $r$ arrows from $n$ to $n-1$. Therefore, $\m$ can be considered as a matrix ring $\begin{bmatrix}\m'&_{\m'}M_{T_2^r(A)}\\0&T_2^r(A)\end{bmatrix}$, where $\m'\cong (1-e_n-e_{n-1})\m(1-e_n-e_{n-1})$, which is isomorphic again to $kQ'/I'$ for some acyclic quiver $Q'$ ($Q'$ is obtained from $Q$ by deleting vertices $n$ and $n-1$.) and monomial ideal $I'$ and $M\cong A\otimes_k N$ for some $kQ'/I'$-$kQ''$-bimodule $N$. According to Theorem~\ref{matrixring_thm}, $\begin{bmatrix} M\otimes X\\X \end{bmatrix}_1$ is a semi-Gorenstein-projective module. However, it is not in $\smon(Q,I,\leftidx{^\perp}A)$ as $\phi$ is not a monomorphism.
\end{proof}

Clearly if $kQ/I$ is semisimple,  $\smon(Q,I,\leftidx{^\perp}A)=\leftidx{^\perp}\m$ holds for any algebra $A$. This exhausts all the possibilities when this equality can hold, which gives a complete answer to Question \ref{que}.

\section {\bf  An example}

In this section, we construct an example of left weakly Gorenstein algebra using Theorem \ref{thm C}. As we have pointed out previously, the class of left weakly Gorenstein algebras is much larger than the class of Gorenstein algebras. Finding left weakly Gorenstein as well as non-left weakly Gorenstein algebras is a fundamental problem.
In \cite[Theorem 1.3]{RZ2}, the authors showed the following criterion:

\begin{thm}
Let $A$ be a finite dimensional algebra. If the number of isomorphism classes of indecomposable left $A$-modules which are both semi-Gorenstein-projective and torsionless is finite, then $A$ is left weakly Gorenstein.
\end{thm}

Due to this result, the class of left weakly Gorenstein algebras already contains a wide class of torsionless finite algebras, as well as the algebras over which the number of isomorphism classes of indecomposable semi-Gorenstein-projective modules is finite.
(Check \cite{Rin} for a list of torsionless finite algebras.)

Beyond this, Theorem \ref{thm C} provides an effective method to construct new left weakly Gorenstein algebras, especially CM-infinite ones, extending the list to even wider classes of algebras. We illustrate this in an example of Nakayama algebra below.

We refer to \cite{Rin2} for a detailed survey into the classification of Gorenstein-projective modules over Nakayama algebras.
Recall that for a Nakayama algebra $A$, the {\it Gorenstein core} $\mathcal C(A)$ is the additive full subcategory consists of indecomposable non-projective Gorenstein-projective $A$-modules as well as their projective covers. It turns out that $\mathcal C(A)$ is always an exact abelian subcategory of $A$-mod \cite[Proposition 1]{Rin2}.

\begin{exm}
Let $A$ be the cyclic Nakayama algebra with Kupisch series $(17,18,18)$, i.e. it is obtained by the following quiver subject to relations: $\{\beta\alpha(\gamma\beta\alpha)^5, (\alpha\gamma\beta)^6\}$.
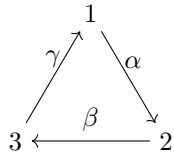
$$
\begin{tikzpicture}[->]
\node(1) at (0,1.7) {$1$};
\node(2) at (1,0) {$2$};
\node (3) at (-1,0) {$3$};

\draw(1)--node [above]{$\ \alpha$}(2);
\draw(2)--node [above]{$\beta$}(3);
\draw(3)--node [above]{$\gamma$}(1);
\end{tikzpicture}
$$

It is easy to check that $A$ is a non-Gorenstein but left weakly Gorenstein algebra.  For the Nakayama algebra $A$ above, the Gorenstein core $\mathcal C(A)\cong k[x]/\langle x^6\rangle$-mod. In fact, indecomposable non-projective Gorenstein-projective $A$-modules are proper quotients of the indecomposable projective module $P(2)$ whose lengths are divisible by $3$.

Consider $Q= \bullet \to \bullet$ and $\m=A\otimes kQ$. Since $A$ is left weakly Gorenstein, so is $\m$. In addition, $\leftidx{^\perp}\m=\Gp(\m)=\smon(Q,0,\Gp(A))\supseteq \smon(Q,0, \mathcal C(A))\cong \mathcal S(k[x]/\langle x^6\rangle)$, the submodule category of $k[x]/\langle x^6\rangle$-mod, which has representation infinite type \cite{RS3}. That is, there are infinitely many isomorphism classes of indecomposable Gorenstein-projective $\m$-modules. Hence $\m$ is a non-Gorenstein but left weakly Gorenstein algebra with infinitely many isomorphism classes of indecomposable left $\m$-modules which are both semi-Gorenstein-projective and torsionless.
\end{exm}

On the other hand, starting from a non-left weakly Gorenstein algebra $A$ (see \cite{JS, RZ2, RZ3, RZ4} for concrete examples), the tensor algebra $A\otimes kQ/I$ will be again non-left weakly Gorenstein.

 \end{document}